\documentclass{amsart}
\usepackage{amssymb,amsmath}
\usepackage{graphicx,ytableau}
\usepackage[enableskew]{youngtab}
\usepackage{young}

\newtheorem{theorem}{Theorem}
\newtheorem{corollary}{Corollary}
\newtheorem{lemma}{Lemma}

\newtheorem{conjecture}{Conjecture}

\begin{document}

\title{Partitions into parts simultaneously regular, distinct, and/or flat}
\author{William J. Keith}
\maketitle

\begin{abstract}
We explore partitions that lie in the intersection of several sets of classical interest: partitions with parts indivisible by $m$, appearing fewer than $m$ times, or differing by less than $m$.  We find results on their behavior and generating functions: more results for those simultaneously regular and distinct, fewest for those distinct and flat.  We offer some conjectures in the area.
\end{abstract}

\section{Introduction}

A partition of $n$ is a nonincreasing sequence of positive integers which sums to $n$, i.e. $\lambda \vdash n$ if $\lambda_1 \geq \lambda_2 \geq \dots \gets \lambda_k$ and $\lambda_1 + \dots + \lambda_k = n$. Their study was initiated by Euler, who proved the usual first result seen by a student of the area, namely

\begin{theorem} The number of partitions of $n$ in which all parts are odd equals the number of partitions of $n$ in which parts are distinct.
\end{theorem}

The theorem was proved by a hands-on combinatorial mapping found by J. J. Sylvester, and then generalized to all moduli by a more general mapping given by his student Glaisher:

\begin{theorem} The number of partitions of $n$ in which no part is divisible by $m$ equals the number of partitions of $n$ in which parts appear fewer than $m$ times.
\end{theorem}

The map of Glaisher's proof can be extended to a general mapping on all partitions: if, given $j$ not divisible by $m$, the part $j m^k$ appears $\sum_{\ell=0}^\infty a_{k,\ell} m^\ell$ times in $\lambda$, written in the base $m$ expansion, then in $\phi(\lambda)$ write the part $j m^\ell$ appearing $a_{k,\ell} m^k$ times for each nonzero $a_{k,\ell}$.  If no part in $\lambda$ is divisible by $m$ ($a_{k,\ell} = 0$ for $k>0$), then in $\phi(\lambda)$ no part will appear $m$ or more times, and vice versa.

The fixed points of the map are precisely those partitions in which parts are not divisible by $m$ (called $m$-\emph{regular} partitions) and in which no part appears $m$ or more times (a partition with the latter property we will call $m$-\emph{distinct}).

The fixed points of an interesting map ought to be of interest, but a search of the literature suggests that little work has been done with these partitions, with the strong exception of the $m=2$ case, partitions into distinct odd parts.  Denoting the number of such partitions of $n$ by $p_{2,2}(n)$, we have that $p_{2,2}(n) \equiv p(n) \pmod{2}$, and since the parity of $p(n)$ is a longstanding question of great interest, $p_{2,2}(n)$ has been much studied, often for its parity properties.

Equal in number with these subsets of partitions of $n$ is the set of those in which the differences between consecutive parts are less than $m$, and the smallest part is less than $m$.  The proof is by \emph{conjugation}, which is defined in terms of the \emph{Ferrers diagram} of a partition: a set of unit squares justified to the origin in the fourth quadrant, in which the $i$-th row below the $x$ axis has $\lambda_i$ squares.  The conjugate of $\lambda$, $\lambda^\prime$, is the partition with Ferrers diagram given by the reflection of the diagram of $\lambda$ across the diagonal.  An example:

\phantom{.}

\begin{center}\begin{tabular}{cc}  $\tiny\young(\hfil\hfil\hfil\hfil,\hfil\hfil\hfil\hfil,\hfil\hfil\hfil:,\hfil:::,\hfil:::,\hfil:::)$ & $\tiny\young(\hfil\hfil\hfil\hfil\hfil\hfil,\hfil\hfil\hfil:::,\hfil\hfil\hfil:::,\hfil\hfil::::)$ \\ $\lambda = (4,4,3,1,1,1) \vdash 14$ & $\lambda^\prime = (6,3,3,2)$
\end{tabular}\end{center}

Now it is easy to see that partitions with parts appearing fewer then $m$ times conjugate to partitions with differences less than $m$ and smallest part less than $m$.  For the remainder of this paper we will call the latter $m$-\emph{flat} partitions, after \cite{Stockhofe}.

\phantom{.}

\noindent \textbf{Remark:} A direct map between $m$-flat and $m$-regular partitions was developed in \cite{Stockhofe}, translated from the German in an appendix to \cite{KThesis}.  (Rather, several involutions on all partitions were constructed, some of which restrict to a map between these sets.)  The fixed points are, again, those that simultaneously satisfy both conditions.

\phantom{.}

Conjugation does not fix those partitions simultaneously $m$-flat and $m$-distinct, but it does fix the class.  In fact, the fixed points of conjugation are in bijection with partitions into distinct odd parts (read vertical-to-horizontal hook lengths).  It might be of interest to develop an involution on partitions which does fix this class; given the utility of conjugation as a theorem-proving tool, its other properties might be of great use.  (If the involution fixes all $m$-flat, $m$-distinct partitions, it will necessarily have at least some other fixed points, as the parity of this subset does not necessairily match that of the number of partitions of $n$.)

In the remainder of the paper we explore those partitions that simultaneously satisfy two of these three conditions, generalizing the question to moduli not necessarily equal for the two conditions.  In Section 2 we discuss partitions simultaneously $s$-regular and $t$-distinct; we can say the most about these.  In Section 3 we discuss $s$-regular, $t$-flat partitions; we can say a few things about these, mostly when $s \vert t$.  In Section 4 we discuss $s$-distinct, $t$-flat partitions; about these we can say little, despite the fact that they have the simplest diagrammatic interpretation.  In the last section we close with some comments and possible lines of future investigation.

\section{Regular and distinct}

The generating function for partitions which are $s$-regular and $t$-distinct is easy to write down: it is \begin{theorem}$$P_{R,D}^{(s,t)}(q) = \sum_{n=0}^\infty p_{R,D}^{(s,t)} (n) q^n = \prod_{k=1}^\infty \frac{(1-q^{sk})(1-q^{tk})}{(1-q^k)(1-q^{stk})}.$$\end{theorem}

$P_{R,D}^{(s,t)} (q)$ is an $\eta$-quotient, i.e. (up to a factor of a power of $q$) a quotient of functions of the form $\eta(z) = q^{1/24} \prod_{n=1}^\infty (1-q^n)$, $q=e^{2\pi i z}$.  By work of Stephanie Treneer \cite{Treneer}, it is known that all such functions are weakly holomorphic modular forms, and so it is likely that they will exhibit many congruences.  Numerical experimentation quickly finds many.  For instance, 

\begin{theorem}\label{PRDCongruences} For $n \geq 0$, 
\begin{align}
p_{R,D}^{(2,2)} (125n+99) &\equiv 0 \pmod{5} \quad \text{(R{\o}dseth)} \\
p_{R,D}^{(3,3)} (4n+2) &\equiv 0 \pmod{2} \\
p_{R,D}^{(2,5)} (4n+3) &\equiv 0 \pmod{2} \quad \text{ and } \\
\sum_{n=0}^\infty p_{R,D}^{(2,5)} (4n+1) q^n &\equiv f_5 \pmod{2}.
\end{align}
\end{theorem}

Here and in the remainder of the paper we employ the shorthand notation $f_k$ for $$f_k = \prod_{n=1}^\infty (1-q^{nk}) = q^{-k/24} \eta(kz) = (q^k;q^k)_\infty.$$  Furthermore, when we state for two power series $f(q) = \sum_{n=n_0}^\infty a(n)q^n$ and $g(q) = \sum_{n=n_1}^\infty b(n) q^n$ that $f(q) \equiv_p g(q)$, we mean that $a(n) \equiv b(n) \pmod{p}$ for all $n$.

\begin{proof}
As noted, the first clause of Theorem \ref{PRDCongruences} was proved by {\O}ystein R{\o}dseth \cite{Rodseth}, who was studying the properties of $p_{2,2}(n)$.

To prove the other clauses we use several identities that dissect various $\eta$-products.  All of the ones we use here can be found in \cite{XiaYao}.  In addition, it is useful to note that $f_k^p \equiv_p f_{kp}$ for $p$ any prime.

For clause (2) we will need: \begin{align}
\frac{{f_3}^3}{f_1} &= \frac{f_4^3f_6^2}{f_2^2f_{12}}+q\frac{f_{12}^3}{f_4} \\
\frac{1}{f_1f_3} &= \frac{f_8^2f_{12}^5}{f_2^2f_4f_6^4f_{24}^2} + q \frac{f_4^5f_{24}^2}{f_2^4f_6^2f_8^2f_{12}}.
\end{align}

Now observe that $$P_{R,D}^{(3,3)}(q) = \frac{f_3^2}{f_1f_9} = \frac{f_3^3}{f_1}\cdot \frac{1}{f_3f_9} = \left( \frac{f_4^3f_6^2}{f_2^2f_{12}}+q\frac{f_{12}^3}{f_4} \right) \left( \frac{f_{24}^2f_{36}^5}{f_6^2f_{12}f_{18}^4f_{72}^2} + q^3 \frac{f_{12}^5f_{72}^2}{f_6^4f_{18}^2f_{24}^2f_{36}} \right).$$

Expanding out the multiplication and reducing modulo 2 where possible, we find $$ P_{R,D}^{(3,3)} (q) \equiv_2 \frac{f_4^3f_{24}^2f_{36}^5}{f_2^2f_{12}^2 f_{18}^4f_{72}^2} + q\left( \dots \right) + q^3\left( \dots \right) + q^4 \frac{f_{12}^8f_{72}^2}{f_4f_6^4f_{18}^2f_{24}^2f_{36}}.$$

The elided terms are all of the form $q^{2n+1}$ and so are irrelevant to the theorem.  Furthermore, neither of the other summands contains powers of the form $q^{4n+2}$ with odd coefficients, since all factors of $f_2$, $f_4$, and $f_{18}$ are raised to even powers, and we may invoke $f_2^2 \equiv_2 f_4$. Hence only powers $q^{4n}$ in these summands may have noneven coefficients, and hence any coefficient $p_{R,D}^{(3,3)}(4n+2) \equiv 0 \pmod{2}$.

For clauses (3) and (4), we additionally require the identity \[ \frac{f_5}{f_1} = \frac{f_8 f_{20}^2}{f_2^2f_{40}} + q \frac{f_4^3f_{10}f_{40}}{f_2^3f_8f_{20}}.\]

Thus

$$P_{R,D}^{(2,5)} (q) = \frac{f_2 f_5}{f_1f_{10}} = \frac{f_8f_{20}^2}{f_2 f_{10}f_{40}} + q \frac{f_4^3 f_{40}}{f_2^2 f_8f_{20}}.$$

Again, since $f_2^2 \equiv_2 f_4$, no term in the latter summand has a noneven coefficient on a power $q^{4n+3}$, and so claim (3) holds.  Further using this identity to reduce the summand, we find that $$P_{R,D}^{(2,5)} (q) \equiv_2 \dots + q f_{20},$$

\noindent where the elided terms are even powers.  Extracting terms of the form $q^{4n+1}$ and making the substitution $q^4 \rightarrow q$, we obtain clause (4), and the theorem holds.

\end{proof}

Many other such congruences can easily be found and proved through similar methods.

\subsection{Symmetry}

Another observation of interest is the symmetry of the generating function, which yields the immediate result

\begin{theorem} The number of partitions of $n$ which are $s$-regular and $t$-distinct equals the number of partitions of $n$ which are $t$-regular and $s$-distinct.
\end{theorem}

It is then reasonable to ask for a map that realizes this equality: as it turns out, if $s$ and $t$ are coprime, a double use of Glaisher's bijection does the job.  Denote by $\phi_m$ Glaisher's involution with modulus $m$.  Then we have the following.

\begin{theorem} If $s$ and $t$ are coprime, then $\phi_s \phi_t$ maps $s$-regular, $t$-distinct partitions to $t$-regular, $s$-distinct partitions.
\end{theorem}

Although this could have been observed earlier, we will see in the midst of this proof that 

\begin{corollary} If $s$ and $t$ are coprime, the number of $s$-regular, $t$-distinct partitions is equal to the number of partitions simultaneously $s$-regular and $t$-regular.
\end{corollary}

However, the $s$-distinct, $t$-distinct partitions are merely the $s$-distinct partitions assuming $s$ is the smaller of the two values.

\begin{proof}
If $s$ and $t$ are coprime, then let $\lambda$ be an $s$-regular, $t$-distinct partition.  The first step $\phi_t$ replaces parts of sizes $j t^k$ with appearances of the part $j$; since $j t^k$ was not divisible by $s$, neither is $j$, and so the result is also an $s$-regular, $t$-regular partition; all such partitions can arise this way ($\phi_s$ or $\phi_t$ reverses the map in the desired direction) and so the corollary follows.  At this point, applying $\phi_s$ produces an $s$-distinct partition which is still $t$-regular, since $j$ is not divisible by $t$ and $j s^k$ is also not divisible by $t$, as these are coprime.

\phantom{.}

\noindent \textbf{Example:} Consider $\lambda = (4,2,1)$ as a 7-regular, 2-distinct partition.  Then $\phi_2((4,2,1)) = (1,1,1,1,1,1,1)$, which is both 7-regular and 2-regular. Then $\phi_7((1,1,1,1,1,1,1)) = (7)$, which is 2-regular and 7-distinct.

\end{proof}

If $s$ and $t$ are not coprime, then, during a visit to Michigan Tech it was conjectured by Bridget Tenner of DePaul University that

\begin{conjecture} Iteration of the previous map suffices to produce a bijection.  That is, there exists $\ell$, varying with $\lambda$, such that $(\phi_s \phi_t)^\ell$ maps an $s$-regular, $t$-distinct partition $\lambda$ to a unique $s$-distinct, $t$-regular partition, with no intervening $(\phi_s \phi_t)^k$ being $s$-regular and $t$-distinct.
\end{conjecture}

Since $\phi_s$ and $\phi_t$ are involutions and the set of partitions of $n$ is finite, the sequence of images $(\phi_s \phi_t)^\ell (\lambda)$ eventually cycles for any $\lambda$; the claim then becomes that such a sequence starting at an $s$-regular, $t$-distinct partition will encounter a $t$-regular, $s$-distinct partition before encountering another $s$-regular, $t$-distinct partition.  (The author must retract a claim made during the presentation at CANT 2016 that the proof of this conjecture is nontrivial but straightforward.  For an indication of the curious behavior that such a sequence can display, the reader might examine the behavior of $(50,50,50,50,50,50)$ as a 6-regular, 10-distinct partition; the map works, but requires 65 iterations, and actually passes through $(50,50,50,50,50,50)$ again halfway through the  63rd step.)

\subsection{McKay-Thompson Series}

For a final comment on the $P_{R,D}^{(s,t)}$ partitions, there is a connection which may be spurious but could be very interesting if it is true in any depth.

To first give some background, recall the $j$-invariant $$j(\tau) = \frac{1}{q} + 196884q + 21493760q^2+\dots .$$

Monstrous Moonshine \cite{Moonshine} is the conjecture, now theorem \cite{MoonshineProof}, that the coefficients of this function are sums of the dimensions of irreducible representations of the Monster group $M$: $1=1$, $196884 = 196883+1$, $21296876+196883+1$, etc.  That is, there is an $\infty$-dimensional graded representation of $M$ whose graded dimension is given by these coefficients, and whose lower-weight pieces decompose into irreps of dimension 1, 196883, 21296876, etc, which sum in fairly simple ways to the coefficients of $j$.  The graded dimension is the graded trace of the identity element $e \in M$; the McKay-Thompson series $T_g$ is the generating function for the graded traces of nontrivial elements $g \in M$.

If we search the invaluable Online Encyclopedia of Integer Sequences \cite{OEIS} for the coefficients of the generating function $P_{(R,D)}^{(3,3)}$, we find that they match OEIS sequence A112194 \cite{OEISMcKay}: ``McKay-Thompson series of class 54c for the Monster group.''  McKay-Thompson series are often of the form $\frac{f_s f_t}{f_1 f_{st}}$, usually shifted by a power of $q$ and with a substitution $q \rightarrow q^\ell$; for instance, the generating function for this McKay-Thompson series is actually $\frac{1}{q} P_{(R,D)}^{(3,3)} (q^6)$.  With a little more searching we find many of these in the OEIS: $(s,t) = (2,5)$ gives class 60F; $(s,t) = (3,4)$ gives 48h; $(s,t) = (5,7)$ gives class 35B, but $(s,t) = (3,7)$ is not there.

So one wonders: is there is a simple, partition-theoretic interpretation of these generating functions in terms of the dimensions being counted?  That is:

\phantom{.}

\noindent \textbf{Question 1:} Are there structures in $M$ or its representations which are in bijection with partitions into, say, partitions into parts not divisible by 2 and appearing less than 5 times, which yield the graded traces of elements in the apparently associated conjugacy classes?

\phantom{.}

Since any $(s,t)$ is a permissible pair for $P_{R,D}^{(s,t)}$, but McKay-Thompson series are restricted by the Monster group itself, such combinatorial descriptions might be ``coincidental''; but, given the great interest in the structure of the Monster group and its subgroups, even descriptions in a few cases might be valuable and interesting in their own right.

\section{Regular and flat}

In this section we discuss partitions simultaneously $s$-regular and $t$-flat.  For these, we can write down the generating function in some restricted cases: namely, when $s \vert t$, much more easily if $s = t$.

We defined $(q;q)_\infty$ earlier; it now becomes useful for us to generalize to the notation $(a;q)_n = \prod_{i=0}^{n-1} (1-a q^i)$, in which case $(q;q)_\infty = \lim_{n\rightarrow \infty} (q;q)_n$.  The empty product is 1, so $(a;q)_0 = 1$.

\subsection{$t$-regular, $t$-flat partitions}

When $s=t$ our task is easiest.  

\begin{theorem} The generating function for partitions both $t$-regular and $t$-flat is $$P_{R,F}^{(t,t)} = \sum_{j=0}^\infty \sum_{i=0}^j\frac{(-1)^i q^{\binom{i+1}{2}t+j-i} (q^{(i+1)t};q^t)_{j-i}}{(q;q)_{j-i}}.$$
\end{theorem}

\begin{proof}

The proof strategy is to note that a $t$-regular partition can be broken into its flat part, plus differences of multiples of $t$:

\begin{center}
\begin{tabular}{cccc}
$a_1$ & $t$ & $t$ & $t$ \\
$a_2$ & $t$ & $t$ & \\
$a_3$ & $t$ & $t$ & \\
$a_4$ & $t$ & $t$ & \\
$a_5$ & & & 
\end{tabular}
\end{center}

\noindent where the $a_i$ are nonzero residues modulo $t$, and each $t$ represents $t$ added to the part.  If $a_{i+1} \leq a_i$, then the number of $t$ units in the flat part of $\lambda_i$ equals the number of such units in $\lambda_{i+1}$, whereas if $a_{i+1} > a_i$, the number of $t$ units in $\lambda_i$ is 1 greater than the number in $\lambda_{i+1}$.  For example, if the above diagram represents the 5-regular partition $(17,13, 11, 11, 4)$, then the flat part of the partition is $(12,8,6,6,4)$.  An amount $5$ was added to parts 1 through 4.  Notice that the $t$-flat part of a $t$-regular partition is still $t$-regular; more generally, the $s$-flat part of a $t$-regular partition is still $t$-regular if $t$ divides $s$.

The amounts added will be multiples of $t$ of sizes up to $t$ times the number of parts of the partition; thus, the generating function for $t$-regular partitions with exactly $j$ parts equals the generating function for $t$-flat, $t$-regular partitions with exactly $j$ parts, times the generating function for partitions into multiples of $t$ no larger than $jt$.

Thus, suppressing the $t$ for now and referring only to the generating functions for partitions of the desired type into exactly $j$ parts, we have 

$$P_{R,F}^{\text{(j parts)}} (q) \times \frac{1}{(q^t;q^t)_j} = P_R^{\text{(j parts)}}.$$

Next we must determine the generating function for $t$-regular partitions into exactly $j$ parts.  We do so by considering all partitions of inclusion-exclusion on the number of sizes of parts of $\lambda$ divisible by $t$, obtaining the following: 

\begin{lemma}$$P_R^{\text{(j parts)}} = \sum_{i=0}^j \frac{q^{j-i}}{(q;q)_{j-i}} (-1)^i q^{\binom{i+1}{2}t} \frac{1}{(q^t;q^t)_i}.$$\end{lemma}

The argument is as follows: begin with $j-i$ guaranteed parts of size 1 and add any desired amount; add exactly $i$ sizes of part divisible by $t$, from $t$ to $it$; finally, add additional multiples of $t$ to these parts alone.  Count those in which we guaranteed at least $i$ different sizes of part divisible by $t$ with $(-1)^i$; by inclusion-exclusion, the resulting sum counts exactly those partitions with no part divisible by $t$.

So, combining identities, 

$$P_{R,F}^{(j \text{ parts})} (q) \times \frac{1}{(q^t;q^t)_j} = P_R^{(j \text{ parts})} = \sum_{i=0}^j \frac{q^{j-i}}{(q;q)_{j-i}} (-1)^i q^{\binom{i+1}{2}t} \frac{1}{(q^t;q^t)_i}.$$

Multiplying through, we obtain 

$$P_{R,F}^{(j \text{ parts})} (q) = \sum_{i=0}^j \frac{q^{j-i}}{(q;q)_{j-i}} (-1)^i q^{\binom{i+1}{2}t} (q^{(i+1)t};q^t)_{j-i}.$$

Summing over numbers of parts $j$, we complete the proof.

\end{proof}

An alternative version of this generating function has more terms but is combinatorially interesting.  Observe that, given a vector $\rho$ of nonzero residues modulo $t$, the $t$-flat partition with residues equal to $\rho$ when read in order is uniquely given.  The number of units of size $t$ below residue $\rho_i$ is precisely the number of pairs $(\rho_k, \rho_{k+1})$ with $k \geq i$ for which $\rho_k < \rho_{k+1}$, i.e. the number of ascents in the multiset permutation, identified by $\rho$, of the multiset of residues listed.

\phantom{.}

\noindent \textbf{Example:} Suppose that $t=3$ and that $\rho$ consists of two 1s and 2s each.  The possible partitions are:

\begin{center}\begin{tabular}{cccccc}
$\begin{matrix} 2 \\ 2 \\ 1 \\ 1
\end{matrix} \quad ,$  & 
$\begin{matrix} 2 & 3 \\ 1 & 3 \\ 2 &  \\ 1 & 
\end{matrix} \quad ,$ & 
$\begin{matrix} 2 & 3 \\ 1 & 3 \\ 1 & 3 \\ 2 & 
\end{matrix} \quad ,$ & 
$\begin{matrix} 1 & 3 \\ 2 &  \\ 2 &  \\ 1 & 
\end{matrix} \quad ,$ & 
$\begin{matrix} 1 & 3 & 3 \\ 2 & 3 &  \\ 1 & 3 &  \\ 2 & & 
\end{matrix} \quad ,$ & 
$\begin{matrix} 1 &  3\\ 1 & 3 \\ 2 &  \\ 2 & 
\end{matrix}$
\end{tabular}
\end{center}

\phantom{.}

The $t$-\emph{complement} $\rho^c$ of $\rho$ is the vector $(t+1-\rho_1,\dots,t+1-\rho_k)$; since ascents in $\rho$ map to descents in $\rho^c$, the number of $t$ units depending from the residue vector is easily seen to be the major index of $\rho^c$.  It is well known (see for instance \cite{Stanley}) that $maj(\rho^c)$ is equidistributed with $maj(\rho)$ over all permutations $\rho$ of the same multiset, and that if $\rho$ contains $i_1$ ones, $i_2$ twos, $\dots$, and $i_{t-1}$ residues $t-1$, then the $q$-multinomial coefficient $$\left[ {{i_1 + \dots + i_{t-1}} \atop {i_1, \dots , i_{t-1}}} \right]_q := \frac{(q;q)_{i_1+\dots+i_{t-1}}}{(q;q)_{i_1}\dots(q;q)_{i_{t-1}}}$$

\noindent is the generating function for the major index over all multiset permutations of $\rho$, i.e. $$ \left[ {{i_1 + \dots + i_{t-1}} \atop {i_1, \dots , i_{t-1}}} \right]_q = \sum_{\sigma(\rho)} q^{maj(\sigma(\rho))}$$

\noindent where summation is over all multiset permutations of $\rho$.

Since the units are of size $t$, we find that the generating function for the $t$-regular, $t$-flat partitions with residue vector some permutation of $\rho$, which we may denote by $P_{R,F}^{(t,t;\rho)}(q)$, is given by \begin{theorem}$$P_{R,F}^{(t,t;\rho)}=q^{i_1 + \dots + (t-1) i_{t-1}} \left[ {{i_1 + \dots + i_{t-1}} \atop {i_1, \dots , i_{t-1}}} \right]_{q^t}.$$\end{theorem}

Finally, we note that if our partitions are $s$-regular and $t$-flat with $s$ dividing $t$, a small variation of the previous argument suffices; we are restricted to a subset of the possible residues modulo $t$.  In the first form of the generating function, when constructing $P_{R}^{(j \text{ parts})}$, we additionally include-exclude parts with residues divisible by $s$, producing additional summations.  For instance, if $2s = t$, we have

$$P_R^{(j \text{ parts})} = \sum_{i,k} \frac{q^{j-i-k}}{(q;q)_{j-i-k}} (-1)^i q^{\binom{i+1}{2}t} (-1)^i q^{\binom{i}{2}t+ks} \frac{1}{(q^t;q^t)_i(q^t;q^t)_k}$$

and hence

$$P_{R,F}^{(j \text{ parts})} = \sum_{i,k} \frac{q^{j-i-k}}{(q;q)_{j-i-k}} (-1)^i q^{\binom{i+1}{2}t} (-1)^i q^{\binom{i}{2}t+ks} \frac{(q^t;q^t)_j}{(q^t;q^t)_i(q^t;q^t)_k}.$$

Other than restricting the permissible residue vectors $\rho$, the second form of the generating function is unchanged.

\subsection{Other observations}

Unlike the other two classes discussed in this paper, simple calculation shows us that $s$-regular, $t$-flat partitions are not symmetric in $s$ and $t$.  For instance, $(1,1,1)$ is 3-regular and 2-flat, and also 2-regular and 3-flat; $(2,1)$ is 3-regular and 2-flat, but not $2$-regular; and $(3)$ is in neither class.  Comparatively, it appears to be the case that the number of $s$-regular, $t$-flat partitions grows faster when $s < t$ than when $s > t$.  An extreme example is the $2$-regular, $t$-flat partitions, which are partitions into odd parts not differing by too much, whereas the $s$-regular, 2-flat partitions can only be partitions into consecutive parts up to size $s-1$.  The asymptotics of these partitions is unexplored, however.

Letting $P_{R,F}^{(s,t;k)} (q)$ be the generating function for $s$-regular, $t$-flat partitions with largest part at most $k$, we have that $$P_{R,F}^{(s,t;k)} (q) = P_{R,F}^{(s,t;k-1)}(q) + \frac{\chi(s \nmid k) q^k}{1-q^k}\left( P_{R,F}^{(s,t;k-1)} - P_{R,F}^{(s,t;k-t)}\right)$$

\noindent where $\chi(T)$ is the indicator function of the truth of statement $T$.

Not many of these generating functions are in the OEIS.  The 2-regular (i.e., partitions into odd parts), 3-flat partitions are partitions into odd parts with consecutive (among odds) sizes, starting with a minimum size of 1; these constitute the mock theta function $\psi(q)$, OEIS seuqence A053251. The 2-regular, 4-flat partitions are the same, except that a 1 need not appear (a 3 always will), and hence $p_{R,F}^{(2,3)}(n) = p_{R,F}^{(2,4)}(n-1)$ for $n>0$.  As mentioned earlier, the $s$-regular, 2-flat partitions are just the partitions into consecutive parts from 1 to $s-1$, such as OEIS sequence A014591.

\section{Distinct and flat}

For some reason we can say very little about partitions simultaneously distinct and flat; except in the most restricted cases, we do not even have a generating function written down for these partitions.  Such observations as can be made are collected below.

Recalling the definition of the Ferrers diagram of a partition, we see that partitions into parts $s$-distinct and $t$-flat can be described geometrically: they are the partitions in which the vertical segments of the outer boundary of the Ferrers diagram -- the \emph{profile} of the partition -- are of length less than $s$, and horizontal segments are of length less than $t$.

It is easy to see from this form that the generating function of the $s$-distinct, $t$-flat partitions is symmetric in $s$ and $t$: the $s$-distinct, $t$-flat partitions of $n$ are in bijection with the $t$-distinct, $s$-flat partitions of $n$ by conjugation.  One notes that the class of $t$-distinct, $t$-flat partitions is preserved, but not the partitions themselves; since the number of $t$-distinct, $t$-flat partitions of $n$ is not necessarily of the same parity as the number of partitions of $n$, it is too much to hope for an involution that has only the $t$-distinct, $t$-flat partitions as its fixed points, but one wonders if there is an involution which at least fixes all of these.

Despite the existence of this simple geometric description, it has been difficult to assert any general form of the generating function.  The $s$-distinct, 2-flat partitions are simply those in which all parts from 1 to some $k$ appear, but at most $s-1$ times.  Their generating function is $$P_{D,F}^{(s,2)} = \sum_{k=0}^\infty q^{\binom{k+1}{2}} \frac{(q^{s-1};q^{s-1})_k}{(q;q)_k}.$$  In particular, the 3-distinct, 2-flat partitions are counted by OEIS sequence A053261, the mock theta function $\psi_1(q)$.

More generally, one can write down various recurrences.  For instance, if $P_{D,F}^{(s,t;k)}(q)$ is the generating function for $s$-distinct, $t$-flat partitions in which the largest part is at most $k$, then $$P_{D,F}^{(s,t;k)}(q) = P_{D,F}^{(s,t;k-1)}(q) + \left( q^k \frac{1-q^{(s-1)k}}{1-q^k} \right) \left( P_{D,F}^{(s,t;k-1)}(q) - P_{D,F}^{(s,t;k-t)}(q) \right)$$ \noindent with appropriate initial conditions.  The standard techniques for solving generating functions, however, do not seem to solve this recurrence very well.

By taking $q \rightarrow 1$ in the previous recurrences, we obtain a solvable difference equation, which can tell us something about the number of such partitions with largest part at most $k$.  For instance, if $s = t = 3$, the simplesst case not covered by the generating function above, we are considering partitions in which parts differ by no more than 2 and repeat no more than twice.  Letting $f(k) = P_{D,F}^{(3,3;k)}(1)$, we find that we have the difference equation $$f(k) = 3 f(k-1) - 2f(k-3),$$

\noindent with initial conditions $f(0) = 1$, $f(1) = 3$, $f(2) = 9$, which yields OEIS sequence A077846, $(1,3,9,25,69,189,517,\dots)$.  At the OEIS entry we find the expression $f(n) = \sum_{i,j = 0}^n 2^j \binom{j}{i-j}$; this is sometimes suggestive of a form for the generating function for a combinatorial expression when one replaces $\binom{N}{M}$ by $\left[ {N \atop M} \right]_{q^k}$ for some useful $k$, but nothing obvious seems to work along these lines for this problem.

The \emph{hooklength} of a square in the Ferrers diagram, identified as position $(i,j)$ when the lower right-hand corner of the square is at $(x,y)$ coordinates $(-i,-j)$ where the upper left corner is the origin, is the sum of the number of squares directly right of and below the square at $(i,j)$, plus 1.  The hooklengths in the partition $(4,4,3,1,1,1)$ are illustrated below.

$$\tiny\young(9542,8431,621:,3:::,2:::,1:::)$$

A partition is $t$-\emph{core} if $t$ is not among its hooklengths.  The partition above is 7-core or $t$-core for $t>9$. Since a partition in which parts differ by $t$ or appear $t$ or more times would automatically have $t$ among the hooklengths in its outermost squares, the $t$-core partitions perforce form a subset of the $t$-distinct, $t$-flat partitions of $n$.  In the case of $t=2$, the sets are equal, as the partitions involved are just the triangular partitions $(n,n-1,\dots,2,1)$.  It might have been hoped that this observation would be useful in producing generating functions, but investigation along this line did not pan out.

\section{Further observations and questions}

Clearly since little can be said about $s$-distinct, $t$-flat partitions, less can possibly be said about partitions simultaneously $r$-regular, $s$-distinct, and $t$-flat.  Those that are 2-regular, 2-distinct and 3-flat are partitions consisting of consecutive odd numbers starting from 1, so their generating function is the Jacobi theta function $\sum_{n=0}^\infty q^{n^2}$.  Those that are 2-regular, 3-distinct and 3-flat permit an additional appearance of each odd part, and these are the 5-th order mock theta function $\phi_0(q)$, OEIS entry A053258.

Several interesting open questions can be posed:

\begin{enumerate}
\item The fact that mock theta functions arise in numerous contexts related to these partitions might be spurious, but after all, a mock theta function has coefficients that do not grow ``too fast,'' and the combination of flatness and another condition restricts partitions rather heavily; while it is perhaps a bit much to hope that the $s$-regular or $s$-distinct and $t$-flat partitions all qualify as mock theta functions, perhaps there is a closer connection here.
\item A full and careful proof of Tenner's conjecture on $\phi_s \phi_t$ for $s$ and $t$ not coprime should be interesting to produce.
\item What is the generating function for partitions with profile segments of length less than 2, that is, into parts appearing not more than twice, with parts differing by at most 2, including starting with 1 or 2?
\item It is easy to show based on Ramanujan's congruences that the number of 5-regular, 5-distinct partitions of $5n+4$ is divisible by 5.  Dyson's rank and the crank do not realize this congruence; is there another natural statistic on this subset which does so?
\end{enumerate}

For item 3, the set of partitions involved is of natural interest, the property is invariant under the most natural involution on partitions, and it has at least a potential relation to the much-studied 3-core partitions, and yet the simple question of writing down the generating function for the set seems to elude any of the basic techniques for doing so.  It would certainly be interesting to see this function written down, and more generally that for the $s$-distinct, $t$-flat partitions.
 
Item 4 is of interest regarding congruences for the partition function such as $p(5n+4) \equiv 0 \pmod{5}$.  One observes that if $p(An+B) \equiv 0 \pmod{C}$ for all $n$, it must also hold that the $p_{A,A}(An+B)$, the number of $A$-regular, $A$-distinct partitions of $An+B$, possesses this congruence, i.e. $p_{A,A}(An+B) \equiv 0 \pmod{C}$.  This follows since one may write a recurrence, perhaps a complicated one but still having integer coefficients, for $p_{A,A}(n)$ in terms of $p(n)$, $p(n-A)$, $p(n-2A)$, etc, and if the latter are all divisible by $C$, then $p_{A,A}(n)$ will be as well.  Since $p_{5,5}(5n+4)$ shares the congruence but the currently constructed statistics fail to realize the congruence, perhaps another statistic exists that does so -- and perhaps, due to the set being considered, is somewhat more natural and susceptible to simpler proof of its properties than the rank and crank.  A really elementary combinatorial proof of Ramanujan's congruences does not yet exist in the literature.

There are certainly many other questions to be explored with these partitions; it is somewhat surprising that they have escaped serious notice for so long, and it is hoped that this paper will spur some interest in this area.

\subsection{Acknowledgements}

The author warmly thanks Melvyn Nathanson and all other organizers and staffers of CANT 2016 for the opportunity to speak and the production of this proceedings volume, and for the lively and interesting discussions and problem sessions which surround the presentations at the conference.

\end{document}